\documentclass[12pt]{article}
\usepackage{hyperref}
\usepackage{cleveref}
\usepackage{amssymb, amsmath, latexsym, a4}
\usepackage{graphicx}
\usepackage[all]{xy}
\usepackage{multicol}
\usepackage[usenames]{color}
 
\usepackage{amssymb}
\usepackage{mathrsfs}
\newtheorem{theorem}{Theorem}[section]
\newtheorem{lemma}[theorem]{Lemma}
\newtheorem{remark}[theorem]{Remark}
\newtheorem{corollary}[theorem]{Corollary}
\newtheorem{proposition}[theorem]{Proposition}
\newtheorem{problem}[theorem]{Problem}
\newtheorem{definition}[theorem]{Definition}
\newtheorem{example}[theorem]{Example}
\newenvironment{proof}{\trivlist\item[]\rm{\textbf{Proof.}\ }}{\endtrivlist}
 \makeatletter
      \def\@setcopyright{}
      \def\serieslogo@{}
      \makeatother

\newcommand{\Ima}{\mathrm{Im}}
\newcommand{\Lie}{\ensuremath{\mathsf{Lie}}}

\author{Hesam Safa  \\ }

\title{The Schur multiplier of an $n$-Lie superalgebra}

\begin{document}
\maketitle


\noindent\textbf{Abstract.} In the present  paper, we  study the notion of the Schur multiplier $\mathcal{M}(L)$
of an $n$-Lie superalgebra $L$,   and prove that $\dim \mathcal{M}(L)\leq  \sum_{i=0}^{n} {m\choose{i}}\mathcal{L}(n-i,k)$,
where $\dim L=(m|k)$, $\mathcal{L}(0,k)=1$ and $\mathcal{L}(t,k)=\sum_{j=1}^{t}{{t-1}\choose{j-1}}{k\choose j}$, for $1\leq t\leq n$.
 Moreover, we obtain an upper bound for
the dimension of  $\mathcal{M}(L)$ in which $L$ is a nilpotent $n$-Lie superalgebra with one-dimensional derived superalgebra.
It is also provided several inequalities on $\dim\mathcal{M}(L)$ as well as an $n$-Lie superalgebra analogue of the converse of Schur's
theorem. \\


\section{Introduction}
 {\it Lie superalgebras} first appeared in a joint work of  F. A. Berezin and G. I. Kac (also written G. I. Kats) in 1970, as  algebras generated by even
(commuting) and odd (anticommuting) variables
\cite{ber},  though it was not them who introduced the name. A year earlier, Victor G. Kac  had started
working on Lie superalgebras after his meeting and discussion with a physicist named  Stavraki who worked on  algebras of field operators \cite{sta2}.
Stavraki showed Kac his example given in \cite{sta}, which was actually the {\it simple Lie superalgebra} $sl(2|1)$.
In 1971,  Kac \cite{kac} published his  first results on  Lie superalgebras or as Stavraki would have liked to call them
{\it graded Lie algebras} \cite{kac2}. Since then, many mathematicians and physicists began working on Lie superalgebras
such that this concept  became a topic of interest in algebra and mathematical physics.

Also, the notion of {\it $n$-Lie superalgebras} (also known as {\it Filippov superalgebras}) was introduced by Daletskii and Kushnirevich \cite{dal} in 1996,
as a natural generalization of  $n$-Lie algebras   \cite{fil}.

The concept of the {\it  Schur multiplier} of a group  originated from a work by Schur \cite{s} on projective representations in
1904. Nearly a century later, Stitzinger and his PhD students Batten and Moneyhun  introduced the Schur multiplier of a Lie algebra in \cite{bat,b-m-s,m}.
Recently, Nayak \cite{n0,n1}  has generalized this notion to Lie superalgebras.
She gives some results on the structure of the Schur multiplier of a Lie superalgebra $L$, denoted by $\mathcal{M}(L)$, and
provides several upper bounds for the dimension of $\mathcal{M}(L)$, as well.

In 1994, Moneyhun \cite{m} proved that $\dim \mathcal{M}(L)\leq\frac{1}{2}m(m-1)$, in which $L$ is  an $m$-dimensional Lie algebra.
In the context of $n$-Lie algebras \cite{d-s}, this bound is  $\dim \mathcal{M}(L)\leq{m\choose n}$. In \cite{n1},
 Nayak shows that  if $L$ is a Lie superalgebra of dimension $(m|k)$, then
  $\dim \mathcal{M}(L)\leq \frac{1}{2}\big{(}(m+k)^2 + (k-m)\big{)}$ (see also \cite{liu}).\\

In this article, we study the notion of the Schur multiplier
of a finite dimensional $n$-Lie superalgebra and obtain an upper bound for its dimension,
which widely extends  all above  bounds.
Moreover, we show that if $L$ is a
 nilpotent $n$-Lie superalgebra
 with $\dim L^2=(1|0)$, then
\begin{eqnarray*}
\dim\mathcal{M}(L)&\leq& \sum_{i=0}^{n} \Bigg{[}{p-1\choose{i}}\mathcal{L}(n-i,q)
+ {m-p+1\choose{i}}\mathcal{L}(n-i,k-q)\Bigg{]}\\
&+&  \sum_{i=1}^{n-1} \Bigg{[}\sum_{j=0}^{i}{p-1\choose j}q^{i-j}\  \sum_{j=0}^{n-i}{m-p\choose j}(k-q)^{n-i-j}\Bigg{]}-1,
\end{eqnarray*}
where   $\dim L=(m|k)$, $\dim Z(L)=(p|q)$ and $\mathcal{L}(t,k)$ is the function defined above.
This upper bound simultaneously  generalizes the corresponding  bounds in Lie algebras \cite{n-r}, $n$-Lie algebras \cite{esh} and Lie superalgebras \cite{n1}, as well.
We also discuss some inequalities on the dimension of $\mathcal{M}(L)$ as well as  a result on  the converse of Schur's theorem
in $n$-Lie superalgebras.


\section{Terminology and notation}

Throughout this paper, all (super)algebras are considered over a fixed field $\mathbb{F}$ of characteristic zero.
In this section, we list some terminologies  on $n$-Lie superalgebras  from \cite{guan,k,sun}.

Let $\mathbb{Z}_2=\{0,1\}$ be a field. A $\mathbb{Z}_2$-graded vector space  (or superspace) $V$ is
a direct sum of vector spaces $V_0$ and $V_1$, whose elements are called even and odd, respectively.
Non-zero elements of $V_0\cup V_1$ are said to be homogeneous. For a homogeneous element $v\in V_a$ with
$a\in\mathbb{Z}_2$,  $|v|=a$ is the
degree of $v$. In the sequel,  when the notation $|v|$ appears, it means that $v$ is a
homogeneous element.
A vector subspace $U$ of $V$ is called $\mathbb{Z}_2$-graded vector subspace (or sub-superspace), if $U=U_0\oplus U_1$ where
$U_0=U\cap V_0$ and $U_1=U\cap V_1$.

\begin{definition}\label{def0} \normalfont
A $\mathbb{Z}_2$-graded vector space $L=L_0\oplus L_1$ is said to be an $n$-Lie superalgebra,
if there exists an $n$-linear map $[-,\ldots,-]:L\times\cdots\times L \longrightarrow  L$
such that
\begin{itemize}
\item[$(i)$] $|[x_1, \ldots,x_n]|=\sum_{i=1}^{n}|x_i|$  (modulo 2),
\item[$(ii)$] $[x_1, \ldots,x_i,x_{i+1},\ldots,x_n]=-(-1)^{|x_i||x_{i+1}|} [x_1, \ldots,x_{i+1},x_i,\ldots,x_n]$  (graded  antisymmetric property),
and
\item[$(iii)$] the following graded Filippov-Jacobi identity holds:
\begin{eqnarray*}
&&[x_1,\ldots,x_{n-1},[y_1,\ldots,y_n]]=[[x_1,\ldots,x_{n-1},y_1],y_2,\ldots,y_n] \\
&&+\sum_{i=2}^{n} (-1)^{(\sum_{j=1}^{i-1}|y_j|) (|x_1|+\cdots+|x_{n-1}|)} [y_1,\ldots,y_{i-1},[x_1,\ldots,x_{n-1},y_i],y_{i+1},\ldots,y_n].
\end{eqnarray*}
\end{itemize}
\end{definition}

Note that this is actually the definition of  an $n$-Lie superalgebra of zero parity (see \cite{guan}, for  general case).
Clearly  $n$-Lie algebras and Lie superalgebras are particular cases of $n$-Lie superalgebras.
In fact, the even part $L_0$ of an $n$-Lie superalgebra $L$ is actually an $n$-Lie algebra,
   which means that if  $L_1=0$, then $L$ becomes an $n$-Lie algebra.

A sub-superspace $I$ of an $n$-Lie superalgebra $L$ is said to be a  sub-superalgebra (resp.  graded ideal), if $[I,I,\ldots,I]\subseteq I$
(resp. $[I,L,\ldots,L]\subseteq I$).
Also, the  center and   commutator (or  derived superalgebra) of  $L$ are
$Z(L)=\{z\in L|\ [z,L,\ldots,L]=0\}$ and
$L^2=[L,\ldots,L]=\langle [x_1,\ldots,x_n]|\ x_i\in L\rangle,$
 repectively, which are  graded ideals of $L$.
 An $n$-Lie superalgebra $L$ is said to be  nilpotent of class $c$, if $L^{c+1}=0$ and $L^c\not=0$, where $L^1=L$ and
  $L^{i+1}=[L^i,L,\ldots,L]$, $i\geq 1$.

Let $L$ and $K$ be two $n$-Lie superalgebras. A multilinear map $f:L\to K$ is called a  homomorphism of $n$-Lie superalgebras, if
$f(L_a)\subseteq K_a$ for every $a\in\mathbb{Z}_2$, and $f([x_1,\ldots,x_n])=[f(x_1),\ldots,f(x_n)]$, for every $x_i\in L$
(see \cite{n1,na} for more details).

Let $0\to R\to F\to L\to 0$ be a  free presentation of an $n$-Lie superalgebra $L$.
 We define the  Schur multiplier of $L$ as
\[\mathcal{M}(L)=\dfrac{R\cap F^2}{[R,F,\ldots,F]},\]
which is an abelian $n$-Lie superalgebra,  independent of the choice of the
free presentation of $L$.
Clearly if $n = 2$, then this definition coincides with the notion
of the  Schur multiplier of a Lie superalgebra given in \cite{n1}.

An exact sequence $0\to M\to K\to L\to 0$ of $n$-Lie superalgebras is called
a  central extension (resp.  stem extension) of $L$, if $M\subseteq Z(K)$ (resp. $M\subseteq Z(K)\cap K^2$).
A   stem cover is a stem extension in which $M\cong \mathcal{M}(L)$. In this case, $K$ is said to be a cover of $L$.

Finally, throughout this paper when an  $n$-Lie superalgebra $L=L_0\oplus L_1$ is of dimension $m+k$,  in which $\dim L_0=m$ and
$\dim L_1=k$, we write $\dim L=(m|k)$.


\section{Preliminary and generalization}

In this section, we provide  some preliminary results on $n$-Lie superalgebras whose proofs are  similar to the corresponding results on Lie (super)algebras
given in \cite{n0,n1}. Hence we omit the proofs.

\begin{lemma}\label{lem21}
Let $0\rightarrow R \rightarrow F \stackrel{\pi}{\rightarrow} L\rightarrow 0$
be a free presentation of an $n$-Lie superalgebra $L$ and $0\rightarrow M \rightarrow K \stackrel{\theta}{\rightarrow} \bar{L}\rightarrow 0$
be a central extension of another $n$-Lie superalgebra $\bar{L}$. Then for every homomorphism $\alpha:L\to\bar{L}$,
there exists a homomorphism $\beta:F/[R,F,\ldots,F]\longrightarrow K$ such that $\beta(R/[R,F,\ldots,F])\subseteq M$
and the following diagram is commutative:
\begin{equation*}
\xymatrix{0\ar[r]&\frac{R}{[R,F,\ldots,F]} \ar[r]\ar[d]_{\beta|}&\frac{R}{[R,F,\ldots,F]}\ar[r]^{\ \ \ \bar{\pi}}\ar[d]_{\beta}& L\ar[r]\ar[d]_{\alpha}&0\\
 0\ar[r]&M\ar[r]& K\ar[r]^{\theta}& \bar{L}\ar[r]&0,}
\end{equation*}
\end{lemma}
\begin{proof}
The proof is a direct generalization of \cite[Lemma 3.2]{n0}.
\end{proof}

\begin{theorem}\label{th22}
Let $L$ be an $n$-Lie superalgebra whose Schur multiplier is finite dimensional, and
$0\rightarrow R \rightarrow F \rightarrow L\rightarrow 0$ be a free presentation of $L$.
Then the extension $0\rightarrow M \rightarrow K \rightarrow L\rightarrow 0$ is a stem cover of $L$ if and only if there exists a graded
ideal $S$ in $F$ such that
\begin{itemize}
\item[$(i)$] $K\cong F/S$ and $M\cong R/S$,
\item[$(ii)$] $\frac{R}{[R,F,\ldots,F]}= \mathcal{M}(L)\oplus\frac{S}{[R,F,\ldots,F]}$.
\end{itemize}
\end{theorem}
\begin{proof}
Direct generalization of  \cite[Theorem 3.3]{n0}.
\end{proof}

\begin{corollary}\label{coro23}
Any finite dimensional $n$-Lie superalgebra has at least one cover.
\end{corollary}
\begin{proof}
It follows from the above theorem.
\end{proof}

\begin{proposition}\label{prop24}
Let $0\rightarrow M \rightarrow K \rightarrow L\rightarrow 0$ be a stem extension of a finite dimensional $n$-Lie
superalgebra $L$. Then $K$ is a homomorphic image of a cover of $L$.
\end{proposition}
\begin{proof}
It is a consequence of Lemma \ref{lem21} and Theorem \ref{th22}. One may also see  \cite[Theorem 3.6]{n0}, for case $n=2$.
\end{proof}

\begin{lemma}\label{lem25}
Let $L$ be a (finite dimensional) $n$-Lie superalgebra and $N$ be a graded ideal of $L$.
Then there exists a (finite dimensional) $n$-Lie superalgebra $K$ with a graded ideal $M$ such that
\begin{itemize}
\item[$(i)$] $L^2\cap N\cong K/M$,
\item[$(ii)$] $M\cong \mathcal{M}(L)$,
\item[$(iii)$] $\mathcal{M}(L/N)$ is a homomorphic image of $K$,
\item[$(iv)$] if $N\subseteq Z(L)$, then $L^2\cap N$ is a homomorphic image of $\mathcal{M}(L/N)$.
\end{itemize}
\end{lemma}
\begin{proof}
Straightforward. See also  \cite[Lemma 3.5]{n1}, for  the usual case.
\end{proof}

\begin{proposition}\label{prop26}
Let $0 \to R\to F  \to L\to 0$ be a free presentation of an $n$-Lie superalgebra $L$. Also,
let $N$ be a graded ideal of $L$ and $S$  a graded ideal of $F$ such that $N\cong S/R$. Then the following sequences are exact:
\begin{itemize}
\item[$(i)$] $0 \longrightarrow \frac{R\cap [S,F,\ldots,F]}{[R,F,\ldots,F]}\longrightarrow  \mathcal{M}(L) \longrightarrow
\mathcal{M}(L/N) \longrightarrow\frac{N \cap L^2}{[N,L,\ldots,L]} \longrightarrow 0$,
\item[$(ii)$] $\mathcal{M}(L) \longrightarrow \mathcal{M}(L/N) \longrightarrow \frac{N \cap L^2}{[N,L,\ldots,L]}
\longrightarrow\frac{N}{[N,L,\ldots,L]}\longrightarrow \frac{L}{L^2}\longrightarrow \frac{L}{N+L^2}\longrightarrow 0$,
\item[$(iii)$] $N\otimes^{n-1}\frac{L}{L^2}\longrightarrow  \mathcal{M}(L) \longrightarrow \mathcal{M}(L/N) \longrightarrow
 N \cap L^2  \longrightarrow 0$, provided that $N$ is a central graded ideal of  $L$.
\end{itemize}
\end{proposition}
\begin{proof}
Lie versions of  these exact sequences are well-known  in the literature (see \cite{b-c,sal}). The proof is a simple generalization of them.
\end{proof}

The following corollary is an immediate consequence of the above proposition.
\begin{corollary}\label{coro27}
Under the notation of Proposition \ref{prop26}, if $L$ is a finite dimensional $n$-Lie superalgebra, then
\begin{itemize}
\item[$(i)$] $\mathcal{M}(L)$ is finite dimensional,
\item[$(ii)$] $\dim\mathcal{M}(L/N) \leq \dim\mathcal{M}(L) + \dim\frac{N \cap L^2}{[N,L,\ldots,L]}$,
\item[$(iii)$] $\dim\mathcal{M}(L) + \dim(N \cap L^2) = \dim\mathcal{M}(L/N) + \dim[N,L,\ldots,L] +
\dim\frac{R\cap [S,F,\ldots,F]}{[R,F,\ldots,F]}$,
\item[$(iv)$]  $\dim\mathcal{M}(L) + \dim(N \cap L^2) = \dim\mathcal{M}(L/N) +
\dim\frac{[S,F,\ldots,F]}{[R,F,\ldots,F]}$,
\item[$(v)$] $\dim\mathcal{M}(L) + \dim(L^2) = \dim\frac{F^2}{[R,F,\ldots,F]}$,
\item[$(vi)$] if $\mathcal{M}(L)=0$, then $\mathcal{M}(L/N)\cong\frac{N \cap L^2}{[N,L,\ldots,L]}$,
\item[$(vii)$] $\dim\mathcal{M}(L) + \dim(N \cap L^2) \leq \dim\mathcal{M}(L/N) +
\dim\big{(} N\otimes^{n-1}\frac{L}{L^2}\big{)}$, provided that $N$ is a central graded ideal of  $L$.
\end{itemize}
\end{corollary}


\section{Upper bounds on the dimension of $\mathcal{M}(L)$}

A  well-known result of  Schur  \cite{s}  states that if   the central factor of a group $G$ is finite, then so is
  $G'$, where $G'$ is the commutator subgroup of $G$.
Over half a century after Schur,  Wiegold \cite{w}
 proved that if  $|G/Z(G)|=p^m$, then $G'$ is a $p$-group of order at most $p^{\frac{1}{2}m(m-1)}$.

 Also in Lie algebras, Moneyhun \cite{m}  showed  that if $L$ is a Lie algebra with $\dim L/Z(L)=m$, then $\dim L^2\leq \frac{1}{2}m(m-1)$.
Her argument is very simple. If $\{\bar{x}_1,\ldots,\bar{x}_m\}$ is a basis for $L/Z(L)$, then $L^2$
 can be trivially generated by
 $\{[x_{i},x_{j}]|\ 1\leq i< j\leq m\}$.
  Therefore,  $\dim L^2\leq {m\choose2}$.
 Moreover, if $L$ is an $n$-Lie algebra  such that $\dim L/Z(L)=m$, then one may similarly show that
  the dimension of  $L^2$ is at most  ${m\choose n}$.
Nayak \cite{n1} has recently proved that if $L$ is a Lie superalgebra with $\dim(L/Z(L))=(m|k)$, then
$\dim L^2\leq \frac{1}{2}\big{(}(m+k)^2 + (k-m)\big{)}$.\\

 As the first result, we provide  an upper bound for the dimension of the commutator of an $n$-Lie superalgebra  with finite
dimensional central factor, which generalizes all above bounds.

\begin{theorem}\label{th1}
Let $L$ be an $n$-Lie superalgebra such that $\dim(L/Z(L))=(m|k)$. Then
\[\dim L^2\leq  \sum_{i=0}^{n} {m\choose{i}}\mathcal{L}(n-i,k),\]
where $\mathcal{L}(0,k)=1$ and $\mathcal{L}(t,k)=\sum_{j=1}^{t}{{t-1}\choose{j-1}}{k\choose j}$, for $1\leq t\leq n$.
\end{theorem}
\begin{proof}
Let $L=L_0\oplus L_1$ and $\{\bar{x}_1,\ldots,\bar{x}_m,\bar{y}_1,\ldots,\bar{y}_k\}$ be a basis for $L/Z(L)$, where
$x_1,\ldots,x_m\in L_0$ and $y_1,\ldots,y_k\in L_1$. The graded antisymmetric property of $n$-Lie superalgebras
 implies that the following set generates $L^2$:
\begin{eqnarray*}
\mathcal{B}=&\Big{\{}[x_{i_1},\ldots,x_{i_s},y_{r_1},\ldots,y_{r_t}]\ |\ 0\leq s,t \leq n,\ s+t=n,\\
& 1\leq i_1 < \cdots < i_s\leq m,\ 1\leq r_1 \leq \cdots \leq r_t\leq k\ \Big{\}}.
\end{eqnarray*}
 First, let $t=0$. Then $s=n$ and the number of commutators  $[x_{i_1},\ldots,x_{i_n}]$ in $\mathcal{B}$ is clearly ${m\choose n}$.
 Now, let $t=n$ and $\mathcal{L}_j$ be the number of commutators $[y_{r_1},\ldots,y_{r_n}]$ in $\mathcal{B}$ such that the set $\{y_{r_1},\ldots,y_{r_n}\}$ contains exactly $j$ distinct elements ($1\leq j\leq n$).
Therefore, $\mathcal{L}_1=k={n-1\choose 0}{k\choose 1}$, since we have these $k$ commutators:
 $$[y_1,\ldots,y_1],\ldots, [y_k,\ldots,y_k].$$
 Also  $\mathcal{L}_2={n-1\choose 1}{k\choose 2}$, which is the number  of  commutators of the form:
 \[[y_{r_c},y_{r_d},\ldots,y_{r_d}], \ [y_{r_c},y_{r_c},y_{r_d}\ldots,y_{r_d}],\ldots,[y_{r_c},\ldots,y_{r_c},y_{r_d}],\]
 in which $1\leq r_c < r_d\leq k$.
Using a similar combinatorial argument, one may check that
 $$\mathcal{L}_j={{n-1}\choose{j-1}}{k\choose j}.$$
Now put $\mathcal{L}(n,k)=\mathcal{L}_1+\mathcal{L}_2+\cdots+\mathcal{L}_{n}=\sum_{j=1}^{n}{{n-1}\choose{j-1}}{k\choose j}$. Thus in  case $t=n$,
the number of commutators  $[y_{r_1},\ldots,y_{r_n}]$ in $\mathcal{B}$ is $\mathcal{L}(n,k)$. By applying a similar technique for all $1\leq t\leq n$
and adding case $t=0$, we have
 $$|\mathcal{B}|={m\choose 0}\mathcal{L}(n,k)+{m\choose 1}\mathcal{L}(n-1,k)+ \cdots+{m\choose n-1}\mathcal{L}(1,k)+{m\choose n},$$
 which completes the proof.
\end{proof}

\begin{remark}\label{rem2} \normalfont
Clearly  if $L$ is a Lie algebra i.e. $n=2$ and $k=0$, then $\mathcal{L}(t,k)=0$ for every $1\leq t\leq n$ and hence our bound is equal to $\frac{1}{2}m(m-1)$ which is the Moneyhun's bound \cite{m}. Also if  $L$ is an $n$-Lie algebra i.e.  $n\geq 2$ and $k=0$, then our bound is equal to
${m\choose n}$    given in \cite{d-s}.
Finally if $L$ is a Lie superalgebra i.e.  $n=2$ and $k\geq 1$, then our bound  will be
\begin{eqnarray*}
\dim L^2 \leq  \sum_{i=0}^{2} {m\choose{i}}\mathcal{L}(2-i,k)=\frac{1}{2}\big{(}(m+k)^2 + (k-m)\big{)},\\
\end{eqnarray*}
which is the Nayak's bound \cite{n1}.
\end{remark}

\begin{theorem}\label{th3}
Let $L$ be an $n$-Lie superalgebra of dimension $(m|k)$. Then
\begin{equation*}\label{eq1}
\dim \mathcal{M}(L)\leq  \sum_{i=0}^{n} {m\choose{i}}\mathcal{L}(n-i,k),
\end{equation*}
where $\mathcal{L}(t,k)$ is the function defined in Theorem \ref{th1}. In particular, the equality occurs if and only if $L$ is an abelian $n$-Lie superalgebra.
\end{theorem}
\begin{proof}
Regarding Corollary \ref{coro23}, let $0\rightarrow M \rightarrow K \stackrel{\pi}{\rightarrow} L\rightarrow 0$ be a stem cover of $L$, i.e. $M\subseteq Z(K)\cap K^2$ and $M\cong \mathcal{M}(L)$. Since $\dim(K/Z(K))\leq \dim(K/M)=\dim L=(m|k)$, we get
$$\dim\mathcal{M}(L)=\dim M\leq \dim K^2\leq \sum_{i=0}^{n} {m\choose{i}}\mathcal{L}(n-i,k).$$
Now if the equality holds, then the above inequality implies that $M=K^2$ and hence $L=K/M$ is abelian.

Conversely, assume that $L$ is abelian, $0\rightarrow M \rightarrow K \stackrel{\pi}{\rightarrow} L\rightarrow 0$ is a stem extension of $L$, and   $\{\bar{x}_1,\ldots,\bar{x}_m,\bar{y}_1,\ldots,\bar{y}_k\}$ is a basis for $L=L_0\oplus L_1$, where
$\pi(x_i)=\bar{x}_i\in L_0$ for $1\leq i\leq m$, and $\pi(y_j)=\bar{y}_j\in L_1$ for  $1\leq j\leq k$. Then $x_i\in K_0$ and $y_j\in K_1$.
 Clearly $M=K^2\subseteq Z(K)$ and also $\mathcal{B}$ (given in Theorem \ref{th1}) is contained in $\ker\pi=M$, since $L$ is abelian.
Moreover  $K$, as a vector superspace, can be generated by the set $N\cup M$, where $N$ is the sub-superspace of $K$ generated by the set
 $\{x_1,\ldots,x_m,y_1,\ldots,y_k\}$. Clearly $N\cap M=0$ and hence as  vector superspaces, we have  $K=N\oplus M$, which  will have
  the structure of an $n$-Lie superalgebra, since $M$ is central.
 Moreover as $M=K^2\subseteq Z(K)$,
 $\mathcal{B}$ is actually a basis for $M$ . Thus
 $$\dim M=|\mathcal{B}|=\sum_{i=0}^{n} {m\choose{i}}\mathcal{L}(n-i,k).$$
 On the other hand,  $0\rightarrow K^2 \rightarrow K \stackrel{\pi}{\rightarrow} L\rightarrow 0$ is  a stem extension of $L$ in which
  $K$ is of maximal dimension. Hence by Proposition \ref{prop24}, $K$ is a cover of $L$ and so
 $M\cong\mathcal{M}(L)$, which completes the proof.
\end{proof}

\begin{corollary}\label{coro5}
Let $L$ be an $n$-Lie superalgebra of dimension $(m|k)$. Then
\[\dim \mathcal{M}(L)\leq  \sum_{i=0}^{n} {m\choose{i}}\mathcal{L}(n-i,k)-\dim L^2.\]
\end{corollary}
\begin{proof}
Let $0\to R\to F\to L\to 0$ be a free presentation of $L$.
Clearly, dimension of the central factor of $F/[R,F,\ldots,F]$ is less than or equal to $\dim (F/R)=(m|k)$.
Hence by Theorem \ref{th1}, $\dim (F^2/[R,F,\ldots,F])\leq \sum_{i=0}^{n} {m\choose{i}}\mathcal{L}(n-i,k)$.
Now using Corollary \ref{coro27} (v), we have
$\dim\mathcal{M}(L) + \dim(L^2) = \dim (F^2/[R,F,\ldots,F]),$
which completes the proof.
\end{proof}

\begin{remark}\label{prob4} \normalfont
One interesting problem in the theory of Lie algebras is the characterization of  Lie algebras using  $t(L)=\frac{1}{2}m(m-1)-\dim\mathcal{M}(L)$,
in which $L$ is an $m$-dimensional Lie algebra. The characterization  of nilpotent Lie algebras for $0\leq t(L)\leq 8$ has been studied in
\cite{b-m-s,h,h-s}. Similarly, Green's result \cite{grn}  yielded a lot of interest on the classification of finite $p$-groups by $t(G),$ investigated by several authors (see \cite{brk, lls, gsc,zh}).

Now, let $L$ be an $n$-Lie superalgebra.  A natural question arises whether one could characterize  $(m|k)$-dimensional
 $n$-Lie superalgebras by $t(n,L)$, where
\[t(n,L)=\sum_{i=0}^{n} {m\choose{i}}\mathcal{L}(n-i,k)-\dim\mathcal{M}(L).\]
In \cite{liu}, this question has been answered  for $t(2,L)\leq 2$.
\end{remark}

The following lemma is needed for proving the next main theorem.

\begin{lemma}\label{lem5}
Let $A=A_0\oplus A_1$ and $B=B_0\oplus B_1$ be two finite dimensional $n$-Lie superalgebras. Then
\begin{eqnarray*}
\dim\mathcal{M}(A\oplus B)=\dim\mathcal{M}(A)+\dim\mathcal{M}(B)
+\sum_{i=1}^{n-1} \Bigg{[}\ \sum_{j=0}^{i}{s_0\choose j}s_1^{i-j}\ \sum_{j=0}^{n-i}{t_0\choose j}t_1^{n-i-j}\Bigg{]},
\end{eqnarray*}
where $s_a=\dim(A_a/A^2\cap A_a)$ and $t_a=\dim(B_a/B^2\cap B_a)$, for $a\in \mathbb{Z}_2$.
\end{lemma}
\begin{proof}
Let $0\to N\to H\to A\oplus B\to 0$ be a stem cover of $A\oplus B$. Then $\frac{H}{N}\cong A\oplus B$ and
$\mathcal{M}(A\oplus B)\cong N\subseteq Z(H)\cap H^2$. Consider graded ideals $X$ and $Y$ of $H$ such that
$\frac{X}{N}\cong A$ and $\frac{Y}{N}\cong B$. Then since $H=X+Y$, we have
\[H^2=X^2+Y^2+\sum_{i=1}^{n-1}[\underbrace{X,\ldots,X}_{i\ times},\underbrace{Y,\ldots,Y}_{(n-i)\ times}].\]
One can easily check that
$N=(N\cap X^2)+(N\cap Y^2)+\sum_{i=1}^{n-1}[X,\ldots,X,Y,\ldots,Y]$, and hence by
Lemma \ref{lem25} (iv) we have
\[\dim\mathcal{M}(A\oplus B)=\dim N\leq \dim \mathcal{M}(A)+
\dim \mathcal{M}(B)+\sum_{i=1}^{n-1}\dim[X,\ldots,X,Y,\ldots,Y].\]
Now, define
 \begin{equation}
\begin{array}{rcl}
f:\dfrac{A}{A^2}\times\cdots\times\dfrac{A}{A^2}\times\dfrac{B}{B^2}\times\cdots\times\dfrac{B}{B^2}&\longrightarrow &[X,\ldots,X,Y,\ldots,Y] \\ \\ \nonumber
(\bar{a}_1,\ldots,\bar{a}_i,\bar{b}_1,\ldots,\bar{b}_{n-i})&\longmapsto & [x_{a_1},\ldots,x_{a_i},y_{b_1},\ldots,y_{b_{n-i}}],
\end{array}
\end{equation}
where $x_{a_i}+N\mapsto a_i$ in $\frac{X}{N}\cong A$, and  $y_{b_i}+N\mapsto b_i$ in $\frac{Y}{N}\cong B$.
One may check that $f$ is a surjective multilinear map.
Clearly
 $$\frac{A}{A^2}=\frac{A_0}{A^2\cap A_0}\oplus\frac{A_1}{A^2\cap A_1}$$
  and we have a similar equality for $\frac{B}{B^2}$.
 Therefore
  $$\dim [\underbrace{X,\ldots,X}_{i},\underbrace{Y,\ldots,Y}_{n-i}]=\dim (\Ima f)\leq
  \sum_{j=0}^{i}{s_0\choose j}s_1^{i-j}\ \sum_{j=0}^{n-i}{t_0\choose j}t_1^{n-i-j}$$
 Hence
 \begin{eqnarray*}
\dim\mathcal{M}(A\oplus B)\leq\dim\mathcal{M}(A)+\dim\mathcal{M}(B)
+\sum_{i=1}^{n-1} \Bigg{[}\sum_{j=0}^{i}{s_0\choose j}s_1^{i-j}\ \sum_{j=0}^{n-i}{t_0\choose j}t_1^{n-i-j}\Bigg{]}.
\end{eqnarray*}
For the next main theorem, we only need the above inequality, so we left the rest of proof which is
an $n$-Lie superalgebra analogue of \cite[Theorem 1]{b-m-s} and \cite[Theorem 3.6]{d-s}.
\end{proof}

It is easy to see that if $L$ is a Lie algebra, i.e. $n=2$ and $s_1=t_1=0$, then the above  sigma is equal to $\dim(A/A^2)\dim(B/B^2)$,
which yields a same dimension as in \cite[Theorem 1]{b-m-s}.

In Theorem \ref{th3}, we determined the dimension of $\mathcal{M}(L)$  when $L$ is an abelian $n$-Lie superalgebra.
In the next result, using Theorem \ref{th3} and Lemma \ref{lem5} we give an upper bound for the dimension of $\mathcal{M}(L)$ in which
$L$ is (non-abelian) nilpotent with $(1|0)$-dimensional derived superalgebra. Note that the following  bound is less than or
equal to the bound given in Corollary \ref{coro5} in  case $\dim L^2=1$.

\begin{theorem}\label{th6}
Let $L$ be a  nilpotent $n$-Lie superalgebra  such that $\dim L=(m|k)$, $\dim Z(L)=(p|q)$ and $\dim L^2=(1|0)$. Then
\begin{eqnarray*}
\dim\mathcal{M}(L)&\leq& \sum_{i=0}^{n} \Bigg{[}{p-1\choose{i}}\mathcal{L}(n-i,q)
+ {m-p+1\choose{i}}\mathcal{L}(n-i,k-q)\Bigg{]}\\
&+&  \sum_{i=1}^{n-1} \Bigg{[}\sum_{j=0}^{i}{p-1\choose j}q^{i-j}\ \sum_{j=0}^{n-i}{m-p\choose j}(k-q)^{n-i-j}\Bigg{]}-1,
\end{eqnarray*}
\end{theorem}
\begin{proof}
Since $L$ is   nilpotent with $\dim L^2=1$, we have  $L^3 \varsubsetneqq L^2$ and so $L^3=0$. Hence $L^2\subseteq Z(L)$.
So one may choose a complement $A$ of $L^2$ in $Z(L)$ and a complement $B/L^2$ of $Z(L)/L^2$ in $L/L^2$. One may easily see that
$L=B+Z(L)$ and
$L^2=B^2$. On the other hand, if $x\in L$, then $x=b+a+l$, for some $b\in B$, $a\in A$ and $l\in L^2$.
Hence $x\in A+B$. Moreover if $x\in A\cap B$, then $x+L^2\in (Z(L)/L^2) \cap (B/L^2)=L^2$ and thus $x\in A\cap L^2=0$. Hence
 $A\cap B=0$ and $L\cong A\oplus B$.
Also (since $A$ is abelian) we have $\dim A/A^2=\dim A=\dim Z(L)-(1|0)=(p-1|q)$, (since $B^2=L^2$)
 $\dim B/B^2=\dim B/L^2=\dim L-\dim Z(L)=(m-p|k-q)$ and  $\dim B=(m-p+1|k-q)$.
Now one can use  Theorem \ref{th3} and Lemma \ref{lem5}.
Note that the above  ($-1$) in the theorem appears because of this fact that $B$ is not abelian and hence  $\mathcal{M}(B)$ cannot
reach the maximum dimension in Theorem \ref{th3}. This completes the proof.
\end{proof}

\begin{remark}\label{rem7} \normalfont
Here we show that our bound extends all previous bounds on Lie algebras, Lie superalgebras and $n$-Lie algebras.

In \cite{n-r}, Niroomand and Russo proved  that if $L$ is an $m$-dimensional nilpotent Lie algebra with $\dim L^2=1$, then
\begin{equation}\label{eq7}
\dim\mathcal{M}(L)\leq \frac{1}{2}(m-1)(m-2)+1,
\end{equation}
and the equality holds if and only if $L\cong H(1)\oplus A(m-3)$, where $H(1)$ denotes the Heisenberg Lie algebra of dimension 3 and
$A(m-3)$ is an abelian Lie algebra of dimension $m-3$.

 Now  in Theorem \ref{th6}, if $L$ is a Lie algebra (i.e. $n=2$ and $k=q=0$) then by a routine computation, our bound will be equal to
\[\dim\mathcal{M}(L)\leq  \frac{1}{2}(m-1)(m-2)+(m-p)-1.\]
Note that if $L\cong H(1)\oplus A(n-3)$, then $(m-p)=\dim L-\dim Z(L)=2$, which yields bound (\ref{eq7}).

Also, Nayak \cite{n1} recently generalized bound (\ref{eq7}) for Lie superalgebras. She proves that if $L$ is
a nilpotent Lie superalgebra of dimension $(m|k)$ and $\dim L^2=(1|0)$, then
\[\dim\mathcal{M}(L)\leq \frac{1}{2}(m+k-1)(m+k-2)+k+1,\]
and the equality occurs if and only if $L\cong H(1,0)\oplus A(m-3,k)$, where $H(1,0)$ is the special Heisenberg Lie superalgebra of dimension
$(3|0)$ and $A(m-3,k)$ is an abelian Lie superalgebra of dimension $m+k-3$.

In Theorem \ref{th6}, if $L$ is a Lie superalgebra (i.e. $n=2$), then after a tedious  computation, our bound turns into
\[\dim\mathcal{M}(L)\leq \frac{1}{2}(m+k-1)(m+k-2)+k+(m+k-p-q)-1,\]
and in  case $L\cong H(1,0)\oplus A(m-3,k)$, we have again $(m+k-p-q)=\dim L-\dim Z(L)=2$, which gives the Nayak's bound.

Finally,  Eshrati et al.  \cite{esh} showed that if $L$ is a nilpotent $n$-Lie algebra of dimension $m$ with $\dim L^2=1$, then
\begin{equation}\label{eq8}
\dim\mathcal{M}(L)\leq {m-1\choose n}+n-1,
\end{equation}
and the equality holds if and only if $L\cong H(n,1)\oplus A(m-n-1)$, where $H(n,1)$ is the special Heisenberg $n$-Lie algebra of dimension
$n+1$ and $A(m-n-1)$ is an abelian $n$-Lie algebra of dimension of $m-n-1$.

Similar to above cases,  in Theorem \ref{th6}, if $L$ is an $n$-Lie algebra (i.e. $k=q=0$), then one may check that our bound is
equal to
\[\dim\mathcal{M}(L)\leq {m-1\choose n}+{m-p\choose n-1}-1,\]
and if $L\cong H(n,1)\oplus A(m-n-1)$, then $m-p=dim L-\dim Z(L)=n$, which yields bound (\ref{eq8}).

Note that the above bounds obtained in \cite{n-r,n1,esh}  are also valid when $\dim L^2=r\geq 1$. In fact,
 their general cases are decreasing functions of $\dim L^2$.
\end{remark}

\begin{problem}\label{prob8} \normalfont
As it is discussed in the above remark, in cases of Lie algebras, Lie superalgebras and $n$-Lie algebras, using the notion
of Heisenberg algebras, one could reach the maximum dimension of $\mathcal{M}(L)$ when $\dim L^2=1$.
R. Bai and Meng \cite{b-m} introduced methods to construct Heisenberg $n$-Lie algebras which are actually
one-dimensional central extensions of  abelian $n$-Lie algebras. Moreover, W. Bai and Liu \cite{b-l} studied the cohomology of Heisenberg Lie superalgebras. A Heisenberg superalgebra can be considered  as a subalgebra of Lie superalgebras
of (odd) contact vector fields, which is precisely the negative part with respect
to a certain natural grading.

Now, how should we define the {\it Heisenberg $n$-Lie superalgebra}
(specially when $n$ is odd and considering the Definition \ref{def0} (i)) to reach the maximum dimension in
Corollary \ref{coro5} and then in Theorem \ref{th6} ?
\end{problem}


In what follows, we provide some inequalities on the dimension of the Schur multiplier of a nilpotent $n$-Lie superalgebra.

\begin{proposition}\label{prop9}
Let $L$ be a finite dimensional nilpotent $n$-Lie superalgebra of class $c$. Then for every $i\geq 2$
\[\dim\mathcal{M}(L)\leq\dim\mathcal{M}(L/L^i)+ (\dim L^c)  \Big{(}\sum_{i=0}^{n-1}{s\choose i}t^{n-1-i}-1\Big{)},\]
where $\dim(L/L^2)=(s|t)$.
\end{proposition}
\begin{proof}
Clearly if either $c=1$ or $i>c$, then the result follows. So assume that $2\leq i\leq c$ and use induction on $\dim L$.
If $\dim L=(1|0)$, then $L$ is abelian. Let $\dim L=(0|1)$ and $\{x\}\subseteq L_1$ be a basis for $L$. Then if $n$ is even, by
Definition \ref{def0} (i) we have  $[x,\ldots,x]=0$ and hence $L$ is abelian. If $n$ is odd and $[x,\ldots,x]=x$, then since
$n-1$ is even and $|x_1|+\cdots+|x_{n-1}|=(n-1)|x|=0$ modulo 2,  the graded
Filippov-Jacobi identity implies that $x=nx$ and $x=0$, which is impossible. Hence in this case, $[x,\ldots,x]=0$ and $L$ is abelian as well, and the inequality holds.
Now suppose that $\dim L>1$ and the result holds for any $n$-Lie superalgebra of dimension less
than $\dim L$. Choose a one-dimensional graded ideal $N$ of $L$ contained in $L^c\subseteq Z(L)$.
By Corollary \ref{coro27} (vii) and the technique applied in Lemma \ref{lem5}, we get
\[\dim\mathcal{M}(L)  \leq \dim\mathcal{M}(L/N)  +\sum_{i=0}^{n-1}{s\choose i}t^{n-1-i}-1.\]
Put $\lambda=\sum_{i=0}^{n-1}{s\choose i}t^{n-1-i}-1$. Since $\dim\frac{L/N}{(L/N)^2}=\dim L/L^2$, using induction hypothesis we have
\begin{eqnarray*}
\dim\mathcal{M}(L)&\leq& \dim\mathcal{M}\big{(}\frac{L/N}{L^i/N}\big{)}+\dim (L^c/N)
\lambda+\lambda\\
&=&\dim\mathcal{M}(L/L^i)+(\dim L^c -1)\lambda+\lambda\\
&=&\dim\mathcal{M}(L/L^i)+(\dim L^c) \lambda.\\
\end{eqnarray*}
\end{proof}

\begin{proposition}\label{prop10}
Let $L$ be a finite dimensional nilpotent $n$-Lie superalgebra of class $c\geq 2$. Then for every $2\leq i\leq c$
\[\dim\mathcal{M}(L)+\dim\mathcal{M}(L/L^i)\leq \big{(}\dim L -1 \big{)}\sum_{i=0}^{n-1}{s\choose i}t^{n-1-i},\]
 where $\dim(L/L^2)=(s|t)$.
\end{proposition}
\begin{proof}
Note that since $c\geq 2$, $L$ is not abelian and so $\dim L>1$. Choose a one-dimensional graded ideal $N$ of $L$ contained
in $L^i\cap Z(L)$. Put $\mu=\sum_{i=0}^{n-1}{s\choose i}t^{n-1-i}$. By the above proposition and induction on $\dim L$, we get
\begin{eqnarray*}
\dim\mathcal{M}(L)  &\leq& \dim\mathcal{M}(L/N) + \mu\\
&\leq& (\dim (L/N) -1)\mu-\dim\mathcal{M}(L/L^i) + \mu\\
&=& (\dim L -2)\mu-\dim\mathcal{M}(L/L^i) + \mu\\
&=& (\dim L -1)\mu-\dim\mathcal{M}(L/L^i).\\
\end{eqnarray*}
\end{proof}

Next, we give a result on the dimension of the Schur multiplier of  $n$-Lie algebras, i.e. $\dim L=(m|0)$.

\begin{definition}\label{def11} \normalfont (\cite{goz}).
An $m$-dimensional  $n$-Lie algebra $L$ ($m\geq n$) is called
 {\it filiform} if $\dim L^i=m-n-i+2$, for every $i\geq 2$.
\end{definition}

Clearly a filiform  $n$-Lie algebra
 is  nilpotent of class  $m-n+1$. Since this nilpotency class is  maximal,  such $n$-Lie algebras can be called  nilpotent
of maximal class.

\begin{example}\label{ex12} \normalfont
Consider the $n$-Lie algebra $L$ with a basis $\{x_1,\ldots,x_m\}$ ($m>n$)
and non-zero multiplications $[x_1,\ldots,x_{n-1},x_i]=x_{i+1}$, for $n\leq i\leq m-1$.
Now, since
\[[\underbrace{x_1,\ldots,x_{n-1},\ldots,[x_1,\ldots,x_{n-1},[x_1,\ldots,x_{n-1}}_{(m-n+1)-times\ x_1,\ldots,x_{n-1}},x_n]]]=[x_1,\ldots,x_{n-1},x_m]=0 \]
and
\[[\underbrace{x_1,\ldots,x_{n-1},\ldots,[x_1,\ldots,x_{n-1},[x_1,\ldots,x_{n-1}}_{(m-n)-times},x_n]]]=[x_1,\ldots,x_{n-1},x_{m-1}]=x_m\not=0,\]
we have $L^{m-n+2}=0$ and $L^{m-n+1}\not=0$.  Therefore,  $L$ is a filiform $n$-Lie algebra.
\end{example}

\begin{theorem}\label{th12}
Let $L$ be a finite dimensional  filiform $n$-Lie algebra. Then
\[\dim\mathcal{M}(L)=\dim\mathcal{M}(L/Z(L))+t,\]
where $-1\leq t\leq n-1$.
\end{theorem}
\begin{proof}
Put $\dim L=m$. If $m=n$, then the result follows. Suppose that $m>n$, then $L$ is not abelian. Since $L$ is filiform, we have $L^{m-n+1}\subseteq Z(L)$ and $\dim L^{m-n+1}=1$. Let $0\not=x=[y,l_1,\ldots,l_{n-1}]\in \dim L^{m-n+1}$, for
some $y\in L^{m-n}$ and $l_i\in L$. Then $y\not\in Z(L)$ and hence $L^{m-n+1}\subseteq Z(L)\subsetneqq L^{m-n}$.
Therefore, $Z(L)=L^{m-n+1}$ as $L$ is filiform, and so $\dim (Z(L)\cap L^2)=1$. Also $\dim (L/L^2)=n$. Now, by
Corollary \ref{coro27} (iv) we have
\[\dim\mathcal{M}(L)  = \dim\mathcal{M}(L/Z(L)) +
\dim\frac{[S,F,\ldots,F]}{[R,F,\ldots,F]}-1,\]
where  $0\rightarrow R \rightarrow F \stackrel{\pi}{\rightarrow} L\rightarrow 0$ is a free presentation of $L$ and $Z(L)\cong S/R$.
Similar to Lemma \ref{lem5},
$\varphi:Z(L) \times \frac{L}{L^2}\times\cdots\times \frac{L}{L^2}\longrightarrow \frac{[S,F,\ldots,F]}{[R,F,\ldots,F]}$
given by $\varphi(z,\bar{l}_1,\ldots,\bar{l}_{n-1})= [s,f_1,\ldots,f_{n-1}]+[R,F,\ldots,F]$, where
$\pi(s)=z$ and $\pi(f_i)=l_i$, is a well-defined epimorphism. Then
$\dim Z(L)=1$ and $\dim (L/L^2)=n$ imply that $\dim (\Ima\varphi)\leq {n\choose n-1}=n$. Therefore,
\[\dim\mathcal{M}(L)\leq \dim\mathcal{M}(L/Z(L))+n-1.\]
On the other hand, by Corollary \ref{coro27} (ii)
\[\dim\mathcal{M}(L/Z(L))\leq \dim\mathcal{M}(L)+1.\]
Combining two inequalities gives the result.
\end{proof}

Finally, we discuss a result concerning the converse of Schur's theorem in the context of $n$-Lie superalgebras.
 In \cite{s-v}, it is shown that if  $L$ is an $n$-Lie algebra with finite dimensional derived algebra and finitely generated central factor, then
$\dim (L/Z(L))\leq {k\choose n-1}\dim L^2$,
where $k=d(L/Z(L))$ is the minimal number of generators of $L/Z(L)$.
This is actually an $n$-Lie algebra analogue of the group case given in \cite{n2}.

In the following theorem, we prove it for $n$-Lie superalgebras which clearly extends the above result.
Recall from \cite{sun} that for $x_1,\ldots,x_{n-1}\in L$, the map $ad(x_1,\ldots,x_{n-1}):L\to L$ given by
$ad(x_1,\ldots,x_{n-1})(x)=[x_1,\ldots,x_{n-1},x]$ for all $x\in L$, is a derivation which is called an inner derivation of the
$n$-Lie superalgebra $L$.
\begin{theorem}\label{th13}
Let $L=L_0\oplus L_1$ be an $n$-Lie superalgebra whose derived superalgebra is finite dimensional and $L/Z(L)$ is finitely generated. Also, let
$$\{x_1+Z(L),\ldots,x_s +Z(L),y_1+Z(L),\ldots,y_t+Z(L)\}$$
be a generating set of $L/Z(L)$ such that $x_i\in L_0$ for $1\leq i\leq s$ and $y_j\in L_1$ for $1\leq j\leq t$. Then
\[\dim\big{(}\frac{L}{Z(L)}\big{)}\leq   \sum_{i=0}^{n-1}{s\choose i}t^{n-1-i} \dim L^2.\]
\end{theorem}
\begin{proof}
Let $A=\{x_1,\ldots,x_s\}$ and $B=\{y_1,\ldots,y_t\}$.
 Put  $ad_{u_{i} v_{i}}=ad(u_1,\ldots,u_i,v_1,\ldots,v_{n-1-i})$ such that $u_1,\ldots,u_i$ are $i$ distinct elements of $A$, and $v_1,\ldots,v_{n-1-i}$
 are $n-1-i$ (not necessarily distinct) elements of $B$. Now, define
 \begin{equation}
\begin{array}{rcl}
\psi:\dfrac{L}{Z(L)}&\longrightarrow &  L^2\oplus\cdots\oplus L^2\\ \\ \nonumber
x+Z(L)&\longmapsto & \big{(}\ldots,ad_{u_{i} v_{i}}(x),\ldots \big{)},
\end{array}
\end{equation}
whose codomain is including  $\sum_{i=0}^{n-1}{s\choose i}t^{n-1-i}$ terms of $L^2$.
Clearly $\psi$ is an injective linear map, and this completes the proof.
\end{proof}


Hesam Safa \\
Department of Mathematics, Faculty of Basic Sciences, University of Bojnord, Bojnord, Iran.\\
E-mail address: hesam.safa@gmail.com, \ \ \  h.safa@ub.ac.ir \\

\end{document}